\documentclass[onecolumn,draftcls,letterpaper,11 pt, conference]{ieeeconf} 
\usepackage{times} % assumes new font selection scheme installed
\usepackage{amsmath} % assumes amsmath package installed
\usepackage{amssymb,amsfonts}  % assumes amsmath package installed
\usepackage{latexsym,theorem,epsfig}
\usepackage{graphicx} 
\usepackage{dsfont}
\usepackage{mathtools}
\usepackage{subfigure}
\usepackage{lscape}
\usepackage{color}
\usepackage{float}

\newtheorem{theorem}{Theorem}

\newtheorem{assumption}{Assumption}

{\theorembodyfont{\rmfamily} }
\newtheorem{lemma}{Lemma}

\allowdisplaybreaks
\IEEEoverridecommandlockouts
\DeclareMathOperator*{\argmin}{arg\,min}
\title{Nonasymptotic Convergence Rates for Cooperative Learning \\ Over Time-Varying Directed Graphs}

\author{Angelia Nedi\'{c}, Alex Olshevsky and C\'{e}sar A. Uribe
\thanks{The authors are with the Coordinated Science Laboratory, University of Illinois, 1308 West Main Street, Urbana, IL 61801, USA, \{angelia,aolshev2,cauribe2\}@illinois.edu.
This research is supported partially by the National Science Foundation under
grant no. CCF-1017564 and by the Office of Naval Research under grant
no. N00014-12-1-0998.}  
}

\begin{document}
\maketitle
\begin{abstract}
   We study the problem of distributed hypothesis testing with a network of agents where some agents repeatedly gain access to information about the correct hypothesis.  The group objective is to globally agree on a joint hypothesis that best describes  the observed data at all the nodes. We assume that the agents can interact with their neighbors in an unknown sequence of time-varying directed graphs. Following the pioneering work of Jadbabaie, Molavi, Sandroni, and Tahbaz-Salehi, we propose local 
   learning dynamics which combine Bayesian updates at each node with a local aggregation rule of private agent signals.  
   We show that these learning dynamics drive all agents to the set of hypotheses which best explain the data collected at all nodes as long as the sequence of interconnection graphs is uniformly strongly connected.  
   Our main result establishes a non-asymptotic, explicit, geometric convergence rate for the learning dynamic. 
   \end{abstract}

%\IEEEpeerreviewmaketitle

\section{Introduction}
Recent years have seen a considerable amount of work on the analysis of distributed algorithms. Nonetheless, the study of distributed decision making and computation can be traced back to the classic papers \cite{aum76,bor82,tsi84} from the 70s and 80s. Applications of such algorithms range from opinion dynamics analysis, network learning and inference, cooperative robotics, communication networks, to social as well as sensor networks. It is the latter settings of social and sensor networks which is the focus of the current paper.

Interactions among people produce exchange of ideas, opinions, observations and experiences, on which new ideas, opinions, and observations are generated. Analyzing dynamic model of such processes generates insight into human behavior and produce algorithms useful in the sensor networking context. 

We consider an agent network where agents repeatedly receive information from their neighbors and private signals from an external source, which provide samples from random variable with unknown distribution. The agents would like to collectively agree on a hypothesis (distribution) that best explains the data. 

Initial results on learning in social networks are described in~\cite{ace11}, where local update rules are designed such that it matches the Bayes' Theorem. That is, given a prior and new observations, the agent is able to compute likelihood functions in order to generate a new posterior, see~\cite{mue13}. Nevertheless, a fully Bayesian approach might not be possible in general since full knowledge of neither the network structure nor other agents hypothesis might be available~\cite{gal03}. Fortunately, non-Bayesian methods have been shown successful in learning as well. For example, in~\cite{eps10}, the authors propose a modification of Bayes' rule that accounts for over-reactions or under-reactions to new information. 

In a distributed setting, several groundbreaking papers have described ways agents achieve global behaviors by repeatedly aggregating local information in a network \cite{jad12,rah10,ala04}. For example, in distributed hypothesis testing using belief propagation, convergence and dependence of the communication structure were shown \cite{ala04}. Later, extensions to finite capacity channels, packet losses, delayed communications and tracking where developed \cite{sal06,rah07}. In \cite{rah10}, the authors proved convergence in probability, asymptotic normality of the distributed estimation and provided conditions under which the distributed estimation is as good as a centralized one. Later in~\cite{jad12}, almost sure convergence of a non-Bayesian rule based on arithmetic mean was shown for fixed topology graphs. Extensions to information heterogeneity and asymptotic convergence rates have been derived as well \cite{jad13}. Following~\cite{jad12}, other methods to aggregate Bayes estimates in a network have been explored. In \cite{ban14}, geometric means are used for fixed topologies as well, however the consensus and learning steps are separated. The work in~\cite{liu14} extends the results of~\cite{jad12} to time-varying undirected graphs. In~\cite{sha13}, local exponential rates of convergence for undirected gossip-like graphs are studied. 

In this paper we propose a non-Bayesian learning rule, analyze its consistency and derive a non-asymptotic rate of convergence for time-varying directed graphs. Our first result shows consistency: we show that over time, the protocol learns the hypothesis or set of hypotheses which better explain the data collected by all the nodes. Moreover, our main result provides a geometric, non-asymptotic, and explicit characterization of the rate of convergence which immediately leads to finite-time bounds which scale intelligibly with the number of nodes. 

In a simultaneous independent effort, the authors in \cite{lal14b,sha14} proposed a similar non-Bayesian learning algorithm where a local Bayes update is followed by a consensus step. In \cite{lal14b}, convergence result for fixed graphs is provided and large deviation convergence rates are given, proving the existence of a random time after which the beliefs will concentrate exponentially fast. In \cite{sha14}, similar probabilistic bounds for the rate of convergence are derived for fixed graphs and comparisons with the centralized version of the learning rule are provided.  

This paper is organized as follows. 
In Section~\ref{model} we describe the model that we study and the proposed update rule. In Section~\ref{consistency} we analyze the consistency of the information aggregation and estimation models, while in Section~\ref{rates} 
we establish non-asymptotic convergence rates of the agent beliefs. Some conclusions and future work directions are given in 
Section~\ref{conclusions}.

\subsubsection*{Notation}
Upper case letters represent random variables (e.g. $X_k$), 
and the corresponding lower case letters for their realizations (e.g. $x_k$). Subindex will generally indicate the time index.
We write as $\left[A_k\right]_{ij}$ the $i$-th row and $j$-th column entry of matrix $A_k$. 
We write $A'$ for the transpose of a matrix $A$ and $x'$ for the transpose of a vector $x$. 
We use $I$ for the identity matrix. 
Bold letters represent vectors which are assumed to be column vectors. 
The $i$'th entry of a vector will be denoted by a superscript $i$, i.e.,  
$\textbf{x}_k = \left[x_k^1,\hdots,x_k^n\right]'$.
We write $\mathbf{1}_n$ to denote the all-ones vector of size $n$. 
For a sequence of matrices $\{A_t\}$, we let  
$A_{t_f:t_i} \triangleq A_{t_f} \cdots A_{t_i+1} A_{t_i}$ for all $t_f \geq t_i \geq 0$.% with $A_{t:t+1} = I$.
We terms "almost surely" and "independent identically distributed" are abbreviated by {\it a.s.} and {\it i.i.d.} respectively.

%%%%%%%%%%%%%%%%%%%%%%%%%%%%%%%%
\section{Problem Setup and Main Results}\label{model}
%%%%%%%%%%%%%%%%%%%%%%%%%%%%%%%%

We consider a group of $n$ agents each of which observes 
a random variable at each time step $k=1, 2, 3, \ldots$. 
We use $S_k^i$ to denote the random variable whose samples are observed by agent $i$ at time step $k$. 
We denote the set of outcomes of the random variable $S_k^i$ by $\mathcal{S}^i$, 
and we assume that this set is finite, i.e., 
$\mathcal{S}^i=\{s_1^i, s_2^i, \ldots, s_{m_i}^i\} \ \hbox{for all }i=1,\ldots,n.$ 
Furthermore, we assume that all $S_k^i$ are i.i.d.\ and drawn according to some  
probability distribution $f^i: \mathcal{S}^i \rightarrow [0,1] $.
For convenience, we stack up all $S_k^i$'s into a single random
vector $\mathbf{S}_k$.
%%%%%%%%%%%%%%%%%%%%%%%%%%%%%%%%%%%%%%%%%%%%
% SELF NOTE: So here each agent receives signals governed by different distributions, as if the nature does not treat them equally
% 	Things to be looked at:
%	What happens if these signals are governed by the same distribution? What dour results say for that case?
%%%%%%%%%%%%%%%%%%%%%%%%%%%%%%%%%%%%%%%%%%%%

We assume there is a finite set of hypothesis, $\Theta = \{ \theta_1, \theta_2, \ldots, \theta_m\}$ and there is a probability distribution $l_i\left(\cdot|\theta\right)$ for each agent $i$ and hypothesis $\theta \in \Theta$. Intuitively, we think of $l_i( \cdot | \theta)$  as the probability distribution seen by agent $i$ if hypothesis $\theta$ were true. Note that, it is not required for the agents to have an hypothesis that is exactly equal to the unknown distribution $f^i$. The goal of the agents is to 
agree on an element of $\Theta$ that fits all the observations in the network best (in a technical sense to be described soon). 
%%%%%%%%%%%%%%%%%%%%%%%%%%%%%%%%%%%%%%%%%%%%
% SELF NOTE: 
% 	They may agree on the best $\theta^*$ but they do not necessarily agree on the "distribution" 
%	since the nature treats them differently. 
%	Suppose that the agent signals are all drawn from the same distribution - then agents need not talk to each other if they have
%	some correct models; but if their models are messed-up (incorrect) - they would need to talk to learn $\theta^*$.
%	In this case, they still do not necessarily agree on the "distribution of signals",
% 	since their $\ell_i(\cdot\mid\theta^*)$ can be different and . 
%%%%%%%%%%%%%%%%%%%%%%%%%%%%%%%%%%%%%%%%%%%%%

Agents communicate with their neighbors, this communication is modedeled as a graph $\mathcal{G}_k=\left\{V,E_k\right\}$ composed 
of a node set $V = \left\{1,2, \hdots, n \right\}$ and a set of directed links $E_k$. 

We will refer to probability distributions over $\Theta$ as {\em beliefs} and assume that agent $i$ begins with an initial belief $\mu_0^i$, which we also refer to as 
its \textit{prior distribution} or \textit{prior belief}. 

This paper focuses in the study of the group dynamics wherein, at time $k$, each agent $i$ updates 
its previous belief $\mu_k^i$ to a new belief $\mu_{k+1}^i$ as follows: 
\begin{align}\label{non_bayes_distributed}
\mu_{k+1}^i\left(\theta\right) & = \frac{\prod_{j=1}^n\mu_k^j\left(\theta\right)^{\left[A_k\right]_{ij}}l_i\left(s_{k+1}^i|\theta\right)}{\sum_{p=1}^m \prod_{j=1}^n \mu_k^j\left(\theta_p\right)^{\left[A_k\right]_{ij}}l_i\left(s_{k+1}^i|\theta_p\right)},
\end{align} 
with $[A_k]_{ij}>0$ when $i$ receives information from $j$ at time~$k$, and else $[A_k]_{ij}=0$.

The ``weight matrices'' $A_k$ satisfy some technical connectivity conditions which have been previously used 
in convergence analysis of distributed averaging and other consensus algorithms~\cite{ber89,mor05,jad03}. The assumptions on the communication graph are presented next.

\begin{assumption}\label{ass_graph}
	The graph sequence $\{ \mathcal{G}_k \}$ and the matrix sequence $\{A_k\}$ are such that:
	\begin{enumerate}
		\item $A_k$ is row-stochastic with $\left[A_k\right]_{ij} > 0$ if $\left(j,i\right) \in E_k$. 
		\item $A_k$ has positive diagonal entries, $\left[A_k\right]_{ii}>0$.
		\item If $\left[A_k\right]_{ij}>0$ then $\left[A_k\right]_{ij} > \eta$ for some positive constant $\eta$.
		\item $\left\{\mathcal{G}_k\right\}$ is $B$-strongly connected, i.e.,
		%or $B$-strongly connected. That is, 
		there is an integer $B\ge 1$ such that the graph 
		$\left\{V,\bigcup_{i=kB}^{\left(k+1\right)B-1}E_i\right\}$ is strongly connected for all $k \geq 0$ .
	\end{enumerate}
\end{assumption} 
%Assumption~\ref{ass_graph}.1 corresponds to information exchange which occurs when nodes broadcast their beliefs to out-neighbors: if $(j,i)$ belongs to the graph at time $k$, then node $i$ uses $j$'s belief in its update, but not necessarily vice versa. 

As a measure for the explanatory quality of the hypotheses in the set $\Theta$ we
use the Kullback-Leibler divergence 
between two discrete probability distributions $\textbf{p}$ and $\textbf{q}$:
\begin{align*}
d\left(\textbf{p} \| \textbf{q}\right) & = \sum_{i=1}^n p_i \log\left(\frac{p_i}{q_i}\right).
\end{align*}
Concretely, the quality of hypothesis $\theta_j$ for agent $i$ is measured by the Kullback-Leibler divergence 
$d\left(f^i\left(\cdot\right)\|l_i\left(\cdot | \theta_j\right)\right)$
between the true distribution of the signals $S^i_k$ and the probability distribution $l_i( \cdot | \theta_j)$  
as seen by agent $i$ if hypothesis $\theta_j$ were correct.
We use the following assumption on the agents' best hypotheses.
\begin{assumption}\label{ass_like} 
	The set $\Theta^*$ defined as $\Theta^* \triangleq\bigcap_{i=1}^n \Theta_i$, where 
	$\Theta_i = \argmin\limits_{\theta \in \Theta} d\left(f^i\left(\cdot\right)\|l_i\left(\cdot | \theta \right) \right)$
	for each $i$, is non-empty.
\end{assumption} 

Assumption~\ref{ass_like} is satisfied if there is some ``true state of the world'' $\widehat{\theta} \in \Theta$ such that each 
agent $i$ sees distributions generated according to $\widehat{\theta}$, i.e., $f^i(\cdot) = l_i(\cdot | \widehat{\theta})$. 
However, this need not be the case for Assumption~\ref{ass_like} to hold. 
Indeed, the assumption is considerably weaker as it merely requires 
that the set of hypotheses, which provide the ``best fits'' for each agent, have at least a single element in common. 

We will further require the following assumptions on the initial distribution and the likelihood functions. 
The first of these is sometimes referred to as the Zero Probability Property \cite{gen86}.

\begin{assumption}\label{ass_init}
	For all agents $i=1,\ldots,n$,
	\begin{enumerate}
		\item The prior beliefs on all $\theta^*\in\Theta^*$ are positive, i.e. $\mu_{0}^i\left(\theta^*\right)>0$ for all $\theta^*\in\Theta^*$.
		\item There exists an $\alpha >0$ such that $l_i\left(s^i | \theta \right)  > \alpha$ for all $s^i \in \mathcal{S}^i$ and $\theta \in \Theta$. 
	\end{enumerate}
\end{assumption}
Assumption~\ref{ass_init}.1 can be relaxed to a requirement that all prior beliefs are positive for some $\theta^*\in\Theta^*$.
Both of these conditions are equally complex to be satisfied. 
They can be satisfied by letting each agent have a uniform prior belief,
which is reasonable in the absence of any initial information about the goodness of the hypotheses. 

We now state our first result, which 
asserts that the dynamics in Eq.~\eqref{non_bayes_distributed}
concentrate all agent's believes in the optimal hypothesis set. We provide its proof in Section \ref{consistency}.

\begin{theorem}\label{proof_consistency} Under
	Assumptions~\ref{ass_graph},~\ref{ass_like}, and~\ref{ass_init}, the update rule of Eq.~(\ref{non_bayes_distributed}) 
	has the following property:
	\begin{align*}
	\lim_{k \rightarrow \infty} \mu_k^i\left(\theta\right) =  0 \quad a.s. \quad\forall \theta \notin \Theta^*, ~i = 1, \ldots, n.
	\end{align*} 
\end{theorem}
The result states that the agents' beliefs will concentrate on the set $\Theta^*$ asymptotically as $k\to\infty$. 

Our main result is a non-asymptotic explicit convergence rate, given in the following theorem, proven in Section \ref{rates}. 
\begin{theorem}\label{teo2} 
	Let Assumptions~\ref{ass_graph},~\ref{ass_like}, and~\ref{ass_init} hold. 
	Also, let $\rho\in(0,1)$ be a given error percentile (or confidence value).
	Then, the update rule of Eq. (\ref{non_bayes_distributed}) has the following property:
	there exists an integer $\boldsymbol{N}(\rho)$ such that,
	with probability $1 -\rho$, 
	for all $k\ge \boldsymbol{N}(\rho)$ there holds that for any $\theta\not\in \Theta^*$, 
	\begin{align*}
	&\mu_{k}^i\left(\theta \right) \leq \exp\left( -\frac{k}{2}\gamma_2+ \gamma_1\right)
	\quad\forall  i = 1, \ldots, n,\\
	& \text{where} \ \ \ \boldsymbol{N}(\rho)
	\triangleq \frac{8 \left( \log\left(\alpha \right)\right) ^2  \log\left(\frac{1}{\rho} \right) }{\gamma_2^2} + 1 ,\\
	\gamma_1 & \triangleq \max_{\theta^*\in\Theta^*\atop\theta \notin\Theta^*}
	\left\{\max_{1\le i \le n} \log \frac{\mu_0^i(\theta)}{\mu_0^i(\theta^*)}
	+ \frac{ C}{1-\lambda}\|\boldsymbol{H}\left(\theta\right)\|_1\right\},\cr
	\gamma_2 & \triangleq \frac{\delta}{n}\, \min_ {\theta\notin\Theta^*}
	\| \boldsymbol{H}\left(\theta\right)\|_1 
	\end{align*}
	$$ \left[ \boldsymbol{H}\left(\theta\right) \right]_i = 
	d(f^i(\cdot) || l_i(\cdot ~|~ \theta)) - d\left(f^i \left( \cdot\right)  \|l_i\left(\cdot|\theta^*\right)\right), $$
	with $\alpha$ from Assumption~\ref{ass_init}.2.
	
	The constants $C$, $\delta$ and $\lambda$ satisfy the following relations:\\
	\noindent (1)
	For general $B$-connected graph sequences $\{\mathcal{G}_k\}$, 
	\begin{align*}
	C = 2, & \ \ \ \ \lambda \leq \left(1-\eta^{nB}\right)^{\frac{1}{B}}, \ \ \ \delta \geq \frac{1}{\eta^{nB}}. 
	\end{align*}
	\noindent (2) If every matrix $A_k$ is doubly stochastic,
	\begin{align*}
	C = \sqrt{2}, & \ \ \ \ \lambda = \left(1-\frac{\eta}{4n^2} \right)^{\frac{1}{B}}, \ \ \ \delta = 1.
	\end{align*} 
	\noindent (3) If each $\mathcal{G}_k$ is an undirected graph and each $A_k$ is the 
	lazy Metropolis matrix, i.e. the stochastic matrix which satisfies
	\begin{align*}
	[A_k]_{ij} & = \frac{1}{2 \max(d(i), d(j))} ~~~\mbox{ for all } \{i,j\} \in \mathcal{G}_k, \\
	\text{then} \ \ \ \ \  C & = \sqrt{2},  \ \ \ \ \lambda = 1- \frac{1}{\mathcal{O}(n^2) }, \ \ \ \   \delta = 1.
	\end{align*}
\end{theorem} 

Note that $\boldsymbol{H}\left(\theta\right)$ does not depend on $\theta^*$ since $d\left(f^i \left( \cdot\right)  \|l_i\left(\cdot|\theta^*\right)\right)$ is the same for all $\theta^*$. 

In contrast to the previous literature, this convergence rate is 
not only geometric but also non-asymptotic and explicit in the sense of immediately leading to bounds which scale
intelligible in terms of the number of nodes. 
For example, in the case of doubly stochastic matrices, 
Theorem~\ref{teo2} 
immediately implies that, after a transient time, which scales cubically in the number $n$ of nodes, 
the network will achieve exponential decay to the correct answer with the exponent 
$-\frac{1}{2}\min_{\theta^*\in\Theta^*}\|\boldsymbol{H}(\theta)\|_1/n$.
%In the case of the lazy Metropolis matrix, where  the transient time scales quadratically with $n$.

Now, consider the case when Assumption~\ref{ass_init}.1 is relaxed to the following requirement:
{\it The prior beliefs on some $\theta^*\in\Theta^*$ are positive}
(i.e. $\mu_{0}^i\left(\theta^*\right)>0$ for some $\theta^*\in\Theta^*$ and all $i$).
Then, it can be seen that the Theorem~\ref{teo2} is valid with 
$\max_{\theta^*\in\Theta^*}$ and $\min_{\theta^*\in\Theta^*}$ 
replaced, respectively, by
$\max_{\theta^*\in\widetilde\Theta^*}$ and $\max_{\theta^*\in\widetilde\Theta^*}$,
where $\widetilde\Theta^*\subseteq\Theta^*$ is the set of all $\theta^*\in \Theta^*$ for which all the agents priors 
$\mu_{0}^i$ are positive.

\section{Consistency of the Learning Rule}\label{consistency} 
In this section we prove Theorem~\ref{proof_consistency}, which 
provides a statement about the consistency (see~\cite{doo49,gho97}) of the  distributed estimator given in
Eq.~(\ref{non_bayes_distributed}). 
Our analysis will require some auxiliary results. First, we will recall some results from \cite{ned13} 
about the convergence of a product of row stochastic matrices. 

\begin{lemma}\label{lemma_angelia}
	\cite{ned13,ned09} 
	Under Assumption~\ref{ass_graph}, for a graph sequence $\left\{\mathcal{G}_k\right\}$ and each $t\geq 0$, there is a stochastic vector $\phi_t$ (meaning its entries are nonnegative and sum to one) such that for all $i,j$ and $k\geq t$,
	\begin{align*}
	\left|  \left[A_{k:t}\right]_{ij} - \phi_t^j\right|  \leq C \lambda^{k-t} \ \ \ \ \forall \ k \geq t \geq 0
	\end{align*}
	where $C>0$ and $\lambda \in \left(0,1 \right) $ satisfy the relations described in Theorem~\ref{teo2}.
\end{lemma}

The proof of Lemma \ref{lemma_angelia} may be found in~\cite{ned13}, with the exception of the bounds on $C,\lambda$
for the lazy Metropolis chains which we omit here due to space constraints.

\begin{lemma} \label{lemma_deltabound}
	\cite{ned13} Let the graph sequence $\left\{\mathcal{G}_k\right\}$ satisfy Assumption~\ref{ass_graph}. Define
	\begin{align}\label{eq:defdelta}
	\delta \triangleq\inf_{k\geq 0} \left(\min_{1\leq i\leq n}\left[\mathds{1}'_n  A_{k:0} \right]_i\right).
	\end{align}
	Then, $\delta \geq \eta^{nB}$, and if all $A_k$ are doubly stochastic, then $\delta=1$.
	Furthermore, the sequence
	$\phi_t$ from Lemma \ref{lemma_angelia} satisfies $\phi_t^j \geq \delta/n$ for all $t \geq 0, j = 1, \ldots, n$.
\end{lemma}
Next, we need a technical lemma regarding the weighted average of random variables with a finite variance.
\begin{lemma}\label{waverage} 
	If assumptions \ref{ass_graph}, \ref{ass_like} and \ref{ass_init} hold. Then for a graph sequence $\{\mathcal{G}_k\}$ we have for any $\theta\notin \Theta^*$ and $\theta^*\in\Theta^*$,
	\begin{align*} 
	\lim_{k \rightarrow \infty} 
	\frac{1}{k}\sum_{t=1}^{{ k}} A_{k:t} 
	\mathcal{L}_{t}^{\theta} + \frac{1}{k}{  \sum_{t=1}^k \mathbf{1}_n\boldsymbol{\phi}'_t \boldsymbol{H}\left(\theta\right) } 
	= 0  \ \ \ a.s.
	\end{align*}
	where $\mathcal{L}_t^\theta$ is the random vector with coordinates given by 
	\[ \left[ \mathcal{L}_t^{\theta} \right] _i=  
	\log \frac{l_i\left(S_{t}^i|\theta\right)}{l_i\left(S_{t}^i|\theta^*\right)} \qquad\forall i=1,\ldots,n,\]
	while the vector $\boldsymbol{H}\left(\theta\right)$ has coordinates given by
	$ \left[ \boldsymbol{H}\left(\theta\right)\right]_i = 
	d(f^i(\cdot) || l_i(\cdot ~|~ \theta)) - d\left(f^i (\cdot) \|l_i\left(\cdot|\theta^*\right)\right) $. 
\end{lemma}
\begin{proof}
	Adding and subtracting $\frac{1}{k}{  \sum_{t=1}^k \mathbf{1}_n\boldsymbol{\phi}'_t \mathcal{L}_{t}^{\theta} }$ yields
	\begin{align}\label{eq:rel} 
	\frac{1}{k}\sum_{t=1}^{{ k}}\left(  A_{k:t} 
	\mathcal{L}_{t}^{\theta} + \mathbf{1}_n\boldsymbol{\phi}'_t \boldsymbol{H}\left(\theta\right)\right)  
	=  
	\frac{1}{k}\sum_{t=1}^{{ k}} \left( A_{k:t} -\mathbf{1}_n\boldsymbol{\phi}'_t \right) 
	\mathcal{L}_{t}^{\theta}  
	+\frac{1}{k} {  \sum_{t=1}^k \mathbf{1}_n\boldsymbol{\phi}'_t    \left( \mathcal{L}_{t}^{\theta} +\boldsymbol{H}\left(\theta\right)\right)  }.
	\end{align}
	By Lemma \ref{lemma_angelia}, 
	$\lim_{k \rightarrow \infty} A_{k:t} =\mathds{1_n}\boldsymbol{\phi}_t'$ for all $t\ge0$. 
	Moreover, each of the entries of $\mathcal{L}_{t}^{\theta}$ are upper bounded by Assumption~\ref{ass_like}. 
	Thus, the first term on the right hand side of Eq.~\eqref{eq:rel} goes to zero as we take the limit over $k \rightarrow \infty$. 
	Regarding the second term in Eq.~\eqref{eq:rel}, by the definition of the KL divergence measure, we have that                                                                                                                                                                                                                                                   
	\begin{align*}
	\mathbb{E}\left[\log \frac{l_i\left(S_{t}^i|\theta\right)}{l_i\left(S_{t}^i|\theta^*\right)}\right] 
	& = \sum_{j=1}^{m_i} f^i\left(s_{j}^i\right)   \log \frac{l_i\left(s_{j}^i|\theta\right)}{l_i\left(s_{j}^i|\theta^*\right)} \\
	& = \sum_{j=1}^{m_i} f^i\left(s_{j}^i\right)\log 
	\left( \frac{l_i\left(s_{j}^i|\theta\right)}{l_i\left(s_{j}^i|\theta^*\right)}\frac{f^i\left(s_{j}^i\right)}{f^i\left(s_{j}^i\right)} \right) \nonumber \\
	& = d\left(f^i(\cdot)\|l_i\left(\cdot|\theta^*\right)\right) - d\left(f^i (\cdot) \|l_i\left(\cdot|\theta\right)\right)
	\end{align*}
	or equivalently $\mathbb{E} \left[ \mathcal{L}_t^{\theta} \right] = -\boldsymbol{H}(\theta)$.
	
	Kolmogorov's strong law of large numbers states that if $\left\{X_t\right\}$ is a sequence of 
	independent random variables with variances such that $\sum_{k=1}^{\infty} \frac{{\rm Var}\left(X_k\right)}{k^2} < \infty$,
	then $\frac{1}{n}\sum_{k=1}^n X_k - \frac{1}{n}\sum_{k=1}^n\mathbb{E}\left[ X_k\right] \rightarrow 0$ a.s.
	Let $X_t=\boldsymbol{\phi}'_t \mathcal{L}_t^{\theta}$. Then, by using
	Assumptions \ref{ass_graph} and~\ref{ass_init}.2, it can be seen
	that ${\sup_{t\ge0}  \text{Var}\left({X_t}\right) < \infty }$. 
	
	The result follows by Lemma~\ref{lemma_angelia} and Kolmogorov's strong law of large numbers.
\end{proof}

With Lemma~\ref{waverage} in place, we are ready to prove  Theorem~\ref{proof_consistency}.
The proof of Theorem~\ref{proof_consistency} (and also Theorem~\ref{teo2}) makes use of the following quantities:
for all $i=1,\ldots,n$ and $k\ge0$,
\begin{align}\label{eq:defphi}
\varphi_{k}^i (\theta)\triangleq\log \frac{\mu_{k}^i\left(\theta\right)}{\mu_{k}^i\left(\theta^*\right)}
\qquad\hbox{for all }\theta\in\Theta,\end{align}
defined for any $\theta^*\in\Theta^*$ (dependence on $\theta^*$ is suppressed).
\begin{proof}
	(\textit{Theorem~\ref{proof_consistency}}) 
	Dividing both sides of (\ref{non_bayes_distributed}) by $\mu_{k+1}^i\left(\theta^*\right)$, then using the log function and the definition of ${\varphi_{k}^i (\theta)}$ we obtain:
	%\begin{align*}
	%\log \frac{\mu_{k+1}^i\left(\theta\right)}{\mu_{k+1}^i\left(\theta^*\right)}
	%& = \log \frac{ \prod_{j=1}^n \mu_k^j\left(\theta\right)^{\left[A_k\right]_{ij}}l_i\left(s_{k+1}^i|\theta\right)}{\prod_{j=1}^n \mu_k^j\left(\theta^*\right)^{\left[A_k\right]_{ij}}l_i\left(s_{k+1}^i|\theta^*\right)}.  
	%\end{align*} 
	%, we can write the preceding relation equivalently:
	\begin{align*}
	\varphi_{k+1}^i\left(\theta\right)
	& = \sum_{j=1}^n \left[A_k\right]_{ij}\varphi_k^j\left(\theta\right) 
	+ \log \frac{l_i\left(s_{k+1}^i|\theta\right)}{l_i\left(s_{k+1}^i|\theta^*\right)}. 
	\end{align*}
	
	Stacking up the values $\varphi_{k+1}^i\left(\theta\right)$ over agents $i=1, \ldots, n,$ 
	into a single vector $\boldsymbol{\varphi}_{k+1}\left(\theta\right)$, we can compactly write the preceding relations, as follows:
	\begin{align}\label{e6}
	\boldsymbol{\varphi}_{k+1}\left(\theta\right)
	& = A_k \boldsymbol{\varphi}_k\left(\theta\right) + \mathcal{L}_{k+1}^{\theta}, 
	\end{align} 
	which implies that for all $k\ge0,$
	\begin{align}\label{e7}
	\boldsymbol{\varphi}_{k+1}\left(\theta\right)
	& = A_{k:0} \boldsymbol{\varphi}_0\left(\theta\right) 
	+ \sum_{t=1}^{k} A_{k:t} \mathcal{L}_{t}^{\theta} + \mathcal{L}_{k+1}^{\theta} .  
	\end{align} 
	We add and subtract $\sum_{t=1}^{k} \mathbf{1}_n\phi_t' \boldsymbol{H}\left(\theta\right)$ in Eq.~\eqref{e7}, then
	\begin{align*}
	\boldsymbol{\varphi}_{k+1}\left(\theta\right)
	& = A_{k:0} \boldsymbol{\varphi}_0\left(\theta\right) 
	+ \sum_{t=1}^{{ k}} \left( A_{k:t} \mathcal{L}_t^{\theta} + \mathbf{1}_n\phi_t' \boldsymbol{H}\left(\theta\right) \right)
+   \mathcal{L}_{k+1}^{\theta} -  \sum_{t=1}^{k} \mathbf{1}_n\phi_t' \boldsymbol{H}\left(\theta\right) .
	\end{align*}
	By using the lower bounds on $\phi_t$ described in 
	Lemma~\ref{lemma_deltabound} and the fact that $\boldsymbol{H}(\theta,\theta^*) \geq 0$, we obtain
	\begin{align*}
	\boldsymbol{\varphi}_{k+1}\left(\theta\right)
	&\leq A_{k:0} \boldsymbol{\varphi}_0\left(\theta\right) 
	+ \sum_{t=1}^{{ k}} \left( A_{k:t} \mathcal{L}_t^{\theta} + \mathbf{1}_n\phi_t' \boldsymbol{H}\left(\theta\right) \right)+ \mathcal{L}_{k+1}^{\theta}    -  \frac{\delta}{n}k \| \boldsymbol{H}\left(\theta\right)\|_1 \mathbf{1}_n.                                  \end{align*}
	Therefore, we have
	\begin{align*}
	&\lim_{k \rightarrow \infty} \frac{1}{k} \boldsymbol{\varphi}_{k+1}\left(\theta\right)
	\leq  \lim_{k \rightarrow \infty} \frac{1}{k} A_{k:0} \boldsymbol{\varphi}_0\left(\theta\right) 
	- \frac{\delta}{n} \| \boldsymbol{H}\left(\theta\right)\|_1 \mathbf{1}_n  {  + \lim_{k \rightarrow \infty}\frac{1}{k} \mathcal{L}_{k+1}^{\theta}  } 
	+ \lim_{k \rightarrow \infty} \frac{1}{k} 
	\sum_{t=1}^{{ k}} \left(A_{k:t} \mathcal{L}_t^{\theta} +\mathbf{1}_n \phi_t' \boldsymbol{H}\left(\theta\right)\right).
	\end{align*}
	The first term of the right hand side of the preceding relation converges to zero deterministically. 
	The third term goes to zero as well since $\mathcal{L}_t^{\theta}$ is bounded, and 
	the fourth term converges to zero almost surely by Lemma \ref{waverage}. Consequently,
	\begin{align}\label{limit_as}
	\lim_{k \rightarrow \infty} \frac{1}{k}\boldsymbol{\varphi}_{k+1}\left(\theta\right) \leq 
	- \frac{\delta}{n}\| \boldsymbol{H}\left(\theta\right)\|_1 \mathbf{1}_n\qquad a.s.
	\end{align} Now if $\theta \notin \Theta^*$, then 
	$\boldsymbol{H}(\theta,\theta^*) > 0$ and, thus, 
	$\boldsymbol{\varphi}_k\left(\theta\right)\to-\infty$ almost surely. 
	This implies $\boldsymbol{\mu}_{k}\left(\theta\right) \rightarrow 0$ almost surely.
\end{proof}

\section{Non-Asymptotic Rate of Convergence}\label{rates}
In this section, we prove Theorem \ref{teo2}, which states an explicit rate of convergence for
cooperative agent learning process. 
Before proving the theorem, we will estate an auxiliary lemma that 
provides a bound for the expectation of the random variables $\varphi_k^i\left(\theta\right)$ as defined in Eq.~\eqref{eq:defphi}.
\begin{lemma}\label{bound_vhi}
	Let  $\theta^*\in\Theta^*$ be arbitrary, and consider $\varphi^i_{k}\left(\theta\right)$ as defined in Eq.~\eqref{eq:defphi}.
	Then, for any $\theta\not\in\Theta^*$  we have
	\begin{align*}
	\mathbb{E}\left[ \varphi^i_{k+1}\left(\theta\right)\right] & \leq
	\gamma_1 - (k+1) \gamma_2\quad\hbox{for all $i$ and $k\ge0$},
	\end{align*}
	where $\gamma_1$ and $\gamma_2$ are defined in Theorem \ref{teo2}.
	%	\begin{align*}
	%\gamma_1 & \triangleq \max_{\theta^*\in\Theta^*}
	%\left\{\|\boldsymbol{\varphi}_0(\theta)\|_{\infty} 	
	%+ \frac{ C}{1-\lambda}\|\boldsymbol{H}\left(\theta\right)\|_1\right\},  \cr
	%\gamma_2 & \triangleq \frac{\delta}{n}\, \min_ {\theta\notin\Theta^*}
	%\| \boldsymbol{H}\left(\theta\right)\|_1. 
	%\end{align*}
\end{lemma}
\begin{proof}
	The expected value of Eq.~\eqref{e6} and 
	${\mathbb{E}\left[\mathcal{L}_{k+1}^{\theta}\right]=-\boldsymbol{H}\left(\theta\right)}$, 
	gives
	\begin{align*}
	\mathbb{E}\left[ \boldsymbol{\varphi}_{k+1}\left(\theta\right)\right] 
	= A_k \mathbb{E}\left[ \boldsymbol{\varphi}_k \left(\theta\right) \right] - \boldsymbol{H}\left(\theta\right)
	\end{align*}
	Therefore, by recursion we can see that for all $k\ge0$,
	\begin{align*}\mathbb{E}\left[ \boldsymbol{\varphi}_{k+1}\left(\theta\right)\right] 
	= A_{k:0} \boldsymbol{\varphi}_0(\theta) - \sum_{t=1}^{k} A_{k:t} \boldsymbol{H}\left(\theta\right) 
	- \boldsymbol{H}\left(\theta\right).
	\end{align*}
	By adding and subtracting $\sum_{t=1}^{k} \mathbf{1}_n\phi_t' \boldsymbol{H}\left(\theta\right)$, we obtain
	\begin{align*}
	\mathbb{E}\left[ \boldsymbol{\varphi}_{k+1}\left(\theta\right)\right] 
	& = A_{k:0} \boldsymbol{\varphi}_0(\theta) 
	+ \sum_{t=1}^{k} \left(  \mathbf{1}_n\phi_t'  -A_{k:t}\right) \boldsymbol{H}\left(\theta\right)  - \sum_{t=1}^{k} \mathbf{1}_n\phi_t' \boldsymbol{H}\left(\theta\right)
	- \boldsymbol{H}\left(\theta\right).
	\end{align*}
	We removed the last term of the right hand side in the preceding relation since $\boldsymbol{H}\left(\theta\right)\geq 0$.  Moreover, bounding the entries for the first two terms on the right hand side and using the fact that $A_{k:0}$ is a stochastic matrix, we have that
	\begin{align*}
	\mathbb{E}\left[ \boldsymbol{\varphi}_{k+1}\left(\theta\right)\right] 
	& \le \|\boldsymbol{\varphi}_0(\theta) \|_\infty  \mathbf{1}_n - \sum_{t=1}^{k} \mathbf{1}_n\phi_t' \boldsymbol{H}\left(\theta\right)
	+ \sum_{t=1}^{k} \max_{1\le i,j\le n}|\phi_t^j  -[A_{k:t}]_{ij}| \|\boldsymbol{H}\left(\theta\right)\|_1  \mathbf{1}_n  
	\end{align*}
	Next, we use the upper bound on terms $|\phi_t^j  -[A_{k:t}]_{ij}|$ from Lemma~\ref{lemma_angelia}
	and the lower bound for the entries in $\phi_t$
	as given in Lemma \ref{lemma_deltabound}, and we arrive at the following relation:
	\begin{align*}
	\mathbb{E}\left[ \boldsymbol{\varphi}_{k+1}\left(\theta\right)\right] 
	& \le \|\boldsymbol{\varphi}_0(\theta) \|_\infty  \mathbf{1}_n
	+ \sum_{t=1}^{k} C \lambda^{k-t} \|\boldsymbol{H}\left(\theta\right)\|_1  \mathbf{1}_n  - k \frac{\delta}{n} \|\boldsymbol{H}\left(\theta\right)\|_1\mathbf{1}_n
	\end{align*}
	% Therefore,
	%for all $k\ge0$,
	%\begin{align*}
	%\mathbb{E}\left[ \boldsymbol{\varphi}_{k+1}\left(\theta\right)\right] 
	% & \le \|\boldsymbol{\varphi}_0(\theta) \|_\infty  \mathbf{1}_n
	% + \frac{C}{1-\lambda} \|\boldsymbol{H}\left(\theta\right)\|_1 \mathbf{1}_n  \cr
	% &\qquad - k \frac{\delta}{n} \|\boldsymbol{H}\left(\theta\right)\|_1 \mathbf{1}_n
	% \end{align*}
	and the result follows. 
	%by letting  
	%\begin{align*}
	%\gamma_1 &= \max_{\theta^*\in\Theta^*}
	%\left\{\|\boldsymbol{\varphi}_0(\theta)\|_{\infty} 	
	%+ \frac{ C}{1-\lambda}\|\boldsymbol{H}\left(\theta\right)\|_1\right\},  \cr
	%\gamma_2 &= \frac{\delta}{n}\, \min_ {\theta\notin\Theta^*}
	%\| \boldsymbol{H}\left(\theta\right)\|_1, 
	%\end{align*}
	%and recalling the definition of $\boldsymbol{\varphi}_0(\theta)$.
\end{proof}

The proof of Theorem \ref{teo2} uses the McDiarmid's inequality~\cite{mcd89}. This will provide bounds on the probability that the beliefs exceed a given value $\epsilon$. McDiarmid's inequality is provided below.

\begin{theorem}\label{mcd}
	(McDiarmid's inequality \cite{mcd89})
	Let $\left\lbrace X_t\right\rbrace _{t=1}^k = (X_1,\hdots,X_k)$ 
	be a sequence of independent random variables with $X_t \in \mathcal{X}$. If a function 
	$g: \left\lbrace X_t\right\rbrace _{t=1}^k \rightarrow \mathbb{R}$ has bounded differences, i.e.,
	for all $t$,
	\begin{align*}
	\sup\limits_{X_t \in \mathcal{X}} g\left(\hdots,X_t,\hdots \right) 
	-\inf\limits_{Y_t \in \mathcal{X}} g\left(\hdots,Y_t,\hdots \right)  &\leq c_t
	\end{align*}
	then for any $\epsilon > 0$ and all $k\ge1$,
	{\small
		\begin{align*}\label{mcdeq}
		\mathbb{P}\left(g\left(\left\lbrace X_t\right\rbrace_{t=1}^k  \right) 
		- \mathbb{E}\left[g\left(\left\lbrace X_t\right\rbrace_{t=1}^k  \right) \right] \geq \epsilon \right)
		\leq \exp\left( {\frac{-2\epsilon^2}{\sum_{t=1}^kc_t^2}  }\right) 
		\end{align*}
	}
\end{theorem}
Now, we are ready to prove Theorem~\ref{teo2}.

\begin{proof}
	(\textit{Theorem} \ref{teo2}) First we will express the belief $\mu_{k+1}^i\left(\theta \right)$ in terms of the variable $\varphi_{k+1}^i\left(\theta\right)$. This will allow us to use the McDiarmid's inequality to obtain the concentration bounds.
	By dynamics of Eq.~\eqref{non_bayes_distributed} and Assumption~\ref{ass_init}.1, since
	$\mu_{k+1}^i\left(\theta^*\right)\in(0,1]$ for any $\theta^*\in\Theta^*$, we have
	\begin{align*}
	\mu_{k+1}^i\left(\theta\right) \le \frac{\mu_{k+1}^i\left(\theta\right)}{\mu_{k+1}^i\left(\theta^*\right)}
	=\exp\left( { \varphi^i_{k+1}(\theta)}\right)
	\end{align*}
	Therefore,

		\begin{align*}
		 \mathbb{P}\left(\mu_{k+1}^i\left(\theta \right) 
		\geq \exp\left( { \frac{-k\gamma_2}{2} + \gamma_1 }\right)  \right) 
		&\leq  \mathbb{P}\left(\varphi_{k+1}^i\left(\theta\right) \geq   \frac{-k\gamma_2}{2}  + \gamma_1 \right)  \\
		&   =  \mathbb{P}\left(\varphi_{k+1}^i\left(\theta\right) - \mathbb{E}\left[\varphi_{k+1}^i\left(\theta\right) \right]  \geq   -\frac{k}{2}\gamma_2 + \gamma_1  - \mathbb{E}\left[\varphi_{k+1}^i\left(\theta\right) \right] \right) \\
		&   = \mathbb{P}\left(\varphi_{k+1}^i\left(\theta\right) - \mathbb{E}\left[\varphi_{k+1}^i\left(\theta\right) \right]  \geq   \frac{k}{2}\gamma_2  \right), 
		\end{align*}

	where the last equality follows from Lemma \ref{bound_vhi}.
	
	We now view $\varphi_{k+1}^i\left(\theta\right)$ a function of the random vectors 
	$s_1,\ldots,s_k,s_{k+1}$, see Eq.~\eqref{e7}, where  $s_t=(s_t^1,\ldots,s_t^n)\in\mathcal{S}$ for all $t$.
	Thus, for all $t$ with $1\le t\le k$, we have
	{\small
		\begin{align*}
		\max_{\textbf{s}_t \in \mathcal{S}} \varphi_{k+1}^i\left(\theta\right) - \min_{\textbf{s}_t \in \mathcal{S}} \varphi_{k+1}^i\left(\theta\right) 
		& =  \max_{\textbf{s}_t \in \mathcal{S}} 
		\sum_{j=1 }^{n}\left[  A_{k:t}\right] _{ij} \left[ \mathcal{L}^{\theta}_t \right]_j   -\min_{\textbf{s}_t \in \mathcal{S}}  \sum_{j=1}^{n}\left[  A_{k:t}\right] _{ij} \left[ \mathcal{L}^{\theta}_t\right] _j \\
		&  =  \max_{\textbf{s}_t \in \mathcal{S}} \sum_{j=1}^{n}\left[  A_{k:t}\right] _{ij} \log \frac{l_j\left(s_t^j|\theta \right) }{l_j\left(s_t^j|\theta^* \right)}    -\min_{\textbf{s}_t \in \mathcal{S}}  \sum_{j=1}^{n}\left[  A_{k:t}\right] _{ij} \log \frac{l_j\left(s_t^j|\theta \right) }{l_j\left(s_t^j|\theta^* \right)}  \\
		%& \leq \max_{s_t^i \in \mathcal{S}^i} \log \frac{l_i\left(s_t^i|\theta \right) }{l_i\left(s_t^i|\theta^* \right)} -\min_{s_t^i \in \mathcal{S}^i} \sum_{i=1}^n \log \frac{l_i\left(s_t^i|\theta \right) }{l_i\left(s_t^i|\theta^* \right)} \\
		& \leq  \log \frac{1 }{\alpha} + \log \frac{1}{\alpha}  \\
		& = 2 \log \frac{1}{\alpha}. 
		\end{align*}
	}
	Similarly, from Eq.~\eqref{e7} we can see that
	\[\max_{\textbf{s}_{k+1} \in \mathcal{S}} \varphi_{k+1}^i\left(\theta\right) 
	- \min_{\textbf{s}_{k+1} \in \mathcal{S}} \varphi_{k+1}^i\left(\theta\right) \le 2 \log \frac{1}{\alpha}.\]
	It follows that  $\varphi_{k+1}^i\left(\theta\right)$ has bounded variations and by 
	McDiarmid's inequality (Theorem \ref{mcd}) we obtain the following concentration inequality,
	\begin{align*}
	\mathbb{P}\left(\varphi_{k+1}^i\left(\theta\right) - \mathbb{E}\left[\varphi_{k+1}^i\left(\theta\right) \right]  \geq   \frac{k}{2}\gamma_2  \right) 
	&\leq   \exp\left( { \frac{-\frac{1}{2} \left( k\gamma_2 \right)^2}{\sum_{t=1}^{k+1} \left(  2 \log \frac{1}{\alpha}\right) ^2} } \right)  \\
	&   
	=  \exp\left( { \frac{-\left( k\gamma_2 \right)^2}{ 8 (k+1) \left( \log \frac{1}{\alpha}\right)^2} }\right)  \\
	& 
	\leq   \exp\left( { -\frac{ (k-1)\gamma_2 ^2}{ 8 \left(  \log \alpha\right)^2} }\right) 
	\end{align*}
	Finally, for a given confidence level $\rho$, in order to have $\mathbb{P}\left(\mu_{k}^i\left(\theta \right) \geq \exp\left( { -\frac{k\gamma_2}{2} + \gamma_1 }\right) \right) \leq \rho$ the desired result follows.
\end{proof}
\section{Simulation Results}\label{simulation}
In this section we show simulation results for a group of agents connected over a time-varying directed graph, shown in Figure \ref{grafo}, for some specific weighting matrices. Each agent updates its beliefs according to Eq. (\ref{non_bayes_distributed}).
%{\small
%\begin{align}\label{A3}
%A_{2k} & =\left[\begin{matrix}
%0.2 & 0.8 & 0 & 0 & 0 & 0   \\
%0.7 & 0.1 & 0.1& 0 & 0 & 0.1  \\
%0  & 0.1 & 0.1& 0.8& 0 & 0   \\
%0  & 0  & 0.8& 0.1& 0.1& 0   \\
%0  & 0  & 0 & 0.1& 0.1& 0.8   \\
%0  & 0.1  & 0 & 0 & 0.8& 0.1  
%\end{matrix}\right] \ \ \ \\
%A_{2k+1} &=\left[\begin{matrix}
%1  & 0  & 0 & 0 & 0 & 0   \\
%0  & 0.8 & 0.1& 0 & 0 & 0.1  \\
%0  & 0.1 & 0.1& 0.8& 0 & 0   \\
%0  & 0  & 0.8& 0.1& 0.1& 0   \\
%0  & 0  & 0 & 0.1& 0.1& 0.8   \\
%0  & 0.1  & 0 & 0 & 0.8& 0.1  \end{matrix}\right]. 
%\end{align} 
%}

Note that the graph is such that the edge connecting agent 1 and agent 2 is switching on and off at each time step. Agents 2-6 connecting edges are changing at each time step as well. 

\begin{figure}[H]
	\centering
	\includegraphics[width=0.35\textwidth]{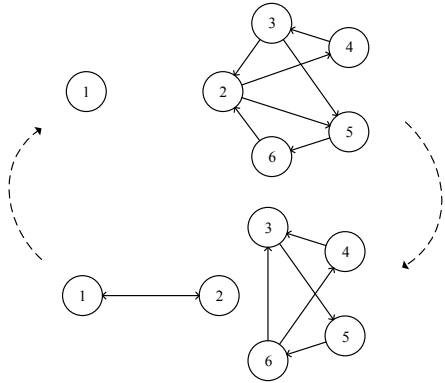}
	\caption{Time-Varying graph with a switching external agent}
	\label{grafo}
\end{figure} 

Every agent $i$ receives information from a binary random variable $S_k^i: \Omega \rightarrow \{0,1\}$ with probability distribution $f^i\left(0\right) = 0.1$ and $f^i\left(1\right) = 0.9$ for all $i$'s. Moreover, every agent has two possible models $\theta_1$ and $\theta_2$. Agent 1 hypotheses have the following likelihood functions: $l_1\left(0 |\theta_1\right)=0.2$ and $l_1\left(1 |\theta_1\right)=0.8$ for hypothesis $\theta_1$; and $l_1\left(0 |\theta_2\right)=0.9$ and $l_1\left(1 |\theta_2\right)=0.1$ for hypothesis $\theta_2$. Therefore, hypothesis $\theta_1$ is closer to the true distribution. On the other hand, agents 2 to 6 have uniformly distributed observationally equivalent hypothesis for both $\theta_1$ and $\theta_2$, that is, they are not able to differentiate between the hypothesis individually. Thus $l_i\left(s|\theta\right) = 0.5$ for $i = \{2,\hdots,6\}$, $s=\{0,1\}$ and $\theta = \{\theta_1,\theta_2\}$. 

Figure \ref{simu_var2} shows the empirical mean over 5000 Monte Carlo simulations of the beliefs on hypothesis $\theta_2$ of agents 1, 4, 5 and 6. Results show that agent 1 is the fastest learning agent, since is the one with the correct model. Nevertheless, all other agents are converging to the correct parameter model as well, even if they do not have differentiable models. 
%\begin{figure}[ht]
%	\centering
%   \includegraphics[width=0.5\textwidth]{Images/t1}
%	\caption{Simulation results for Agents 1, 2 and 3}
%	\label{simu_var1}
%\end{figure}
\begin{figure}[ht]
	\centering
	\includegraphics[width=0.5\textwidth]{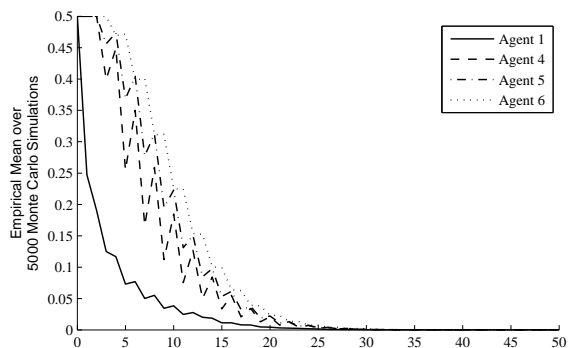}
	\caption{Simulation results for Agents 1, 4, 5 and 6}
	\label{simu_var2}
\end{figure}

\section{Conclusions and Future Work}\label{conclusions}

We have studied the consistency and the rate of convergence
for a distributed non-Bayesian learning system. We have 
shown almost sure consistency and have provided
bounds on the global exponential rate of convergence. The
novelty of our results is in the establishment of convergence rate estimates that are non-asymptotic, geometric, and explicit, in the sense that the bounds capture the quantities characterizing the graph sequence properties as well as the agent learning capabilities. This results were derived for general time-varying directed graphs.

Our work suggests a number of open questions. It is natural to attempt to extensions to continuous spaces, on the number of agents, on the number of hypothesis, etc. This result can be extended to tracking problems where the distribution of the observations changes with time. When the number of hypothesis is large, ideas from social sampling can also be incorporated in this framework \cite{sar13}. Moreover,  the possibility of corrupted measurements or conflicting models between the agents are also of interest, especially in the setting of social networks.

%\section*{Acknowledgment}
%
%The authors would like to thank...

\bibliographystyle{IEEEtran} %using IEEEtr was causing me trouble - Angelia

\bibliography{IEEEfull,bayes_cons}

\end{document}